\documentclass[12pt]{amsart}
\usepackage{graphicx,mathtools}
\usepackage{amsmath,amsthm}
\usepackage{amssymb}
\usepackage{moreverb}
\usepackage{amsfonts,latexsym,graphicx,url}
\usepackage{empheq}
\usepackage{mathrsfs}
\usepackage{cases}
\usepackage{hyperref}
\usepackage{xcolor}
\usepackage{breqn}
\usepackage{amsaddr} 
\usepackage[pagewise]{lineno}
 \usepackage{mathptmx}      
 \usepackage[hmarginratio=1:1]{geometry}
\theoremstyle{plain}
\newtheorem{theorem}{Theorem}[section]

\newtheorem{lemma}[theorem]{Lemma}
\newtheorem{proposition}[theorem]{Proposition}
\newtheorem{definition}[theorem]{Definition}
\newtheorem{remark}[theorem]{Remark}

\newcommand{\E}{{\mathbf E}}
\newcommand{\B}{{\mathbf H}}
\def \RR {\mathbb{R}}
\def \x {{\mathbf x}}
\def \u {{\mathbf u}}
\def \v {{\mathbf v}}
\def \m {{\mathbf m}}
\def \T {\textbf{T}(\partial \Omega)}
\newcommand{\Hcurl}{\textbf{H}\left(\text{curl},\Omega\right)}
\newcommand{\iu}{{i\mkern1mu}}
\newtheorem{example}[theorem]{Example}

\newcommand\BibTeX{{\rmfamily B\kern-.05em \textsc{i\kern-.025em b}\kern-.08em
T\kern-.1667em\lower.7ex\hbox{E}\kern-.125emX}}

\begin{document}

\title[Quantitative estimates in Lipschitz domains]{Quantitative Trace Estimates for the Maxwell system in Lipschitz Domains}

\author{Eric Stachura}
 \address{Kennesaw State University\\
          Department of Mathematics\\
            850 Polytechnic Lane\\
            Marietta, GA 30060 USA \\}
\email{estachur@kennesaw.edu}             
 \author{Niklas Wellander}
     \address{
          Swedish Defense Research Agency, FOI\\
          Division of Command and Control Systems\\
           SE - 581 11  Link\"oping, Sweden\\}    
            \email{niklas.wellander@foi.se}
        \thanks{Corresponding author: E. Stachura}

\maketitle

\begin{abstract} We develop various quantitative estimates for the anisotropic Maxwell system in Lipschitz domains, with a focus on how the estimates precisely depend on the Lipschitz character of the domain. We pay special attention to trace operators and extension operators over certain Sobolev spaces. Finally, we provide a weak formulation of the interior scattering problem in terms of the exterior Calder\'on operator, and provide explicit bounds for the solution of the interior problem in terms of the incident fields and the Lipschitz character of the domain.
\end{abstract}

\smallskip

\noindent \textbf{Keywords.} Trace Theorems; Maxwell Equations; Lipschitz domains; Scattered fields; Electromagnetic Scattering

\section{Introduction}
Estimates of scattered fields in terms of incident fields play an important role in controlling how much energy an obstacle scatters. In \cite{NWGK2014} this problem was addressed quite explicitly, and a number of such estimates were obtained on a domain with $C^{1,1}$ boundary. Here we instead consider the case when the domain is rougher; namely, the domains under consideration here are Lipschitz. 
 
 The effects of rough surfaces on scattering are important for many applications, including radar surveillance, remote sensing, synthetic aperture radar imaging and radio communication, see e.g. \cite{Beckmann1987} and \cite{Chinese2013}.
 
 It is our motivation to study precisely \emph{how} the rough surface effects the various scattering problems. We develop quantitative estimates for various trace operators, which are used in turn to prove existence and uniqueness of solutions to an interior problem for the Maxwell system. The previous trace estimates are then used to establish an $\Hcurl$ estimate of the solution of the interior problem, where the constant is obtained explicitly in terms of the Lipschitz nature of the domain. Indeed, we obtain the following explicit bound of the field $\E$ in terms of the incident fields:

 \begin{theorem}\label{main1}
 Any solution $\E\in \Hcurl$ of (\ref{weakformulation}) on a bounded, Lipschitz domain in $\mathbb{R}^3$ satisfies
 \begin{align*}
||\E||_{\textbf{H}(\text{curl},\Omega)} \leq \dfrac{(1+\sqrt{2}) (M\;k_1)^2}{\min(C_0, \widetilde{C_0})}||C^e||_{\T} \left( ||\B_i||_{\textbf{H}(\text{curl},\Omega)} + ||\E_i||_{\textbf{H}(\text{curl},\Omega)}\right),
\end{align*}
where the constants $C_0, \widetilde{C_0}$ depend on the material parameters of the scatterer, $k_1$ is a constant given by (\ref{k1def}), and $M$ is the Lipschitz constant of the domain.
 \end{theorem}
 
 The Sobolev spaces employed in this general setting are more complicated than on a smooth domain, see \cite{buffa2002} and the references therein. This is exemplified by the fact that even for $u\in C^\infty (\overline{\Omega})$, it is not true that the tangential trace of $u$ belongs to $H^{1/2}(\partial \Omega)$. Therefore, the development of precise, qualitative trace estimates is important for understanding exactly how the rough surface affects solutions of the Maxwell system, and ultimately how the scattering problem is affected. 
 
 We expect the quantitative trace estimates themselves to be of independent interest, but in particular these results will be useful for the exploration of how precisely rough surfaces effect various electromagnetic scattering problems. 
 
 The layout of this article is as follows. First in Section \ref{sec:Lipdomain} we recall the basics of Lipschitz domains. In particular, we define a Lipschitz character $\theta$ in (\ref{thetadef}) that will appear in a number of trace estimates. Section \ref{sec:traces} is devoted to proving the quantitative trace theorems. In Section \ref{sec:scattering} we consider the scattering problem. Then in Section \ref{sec:weakformulation} we develop a weak formulation of the scattering problem, prove existence of weak solutions via Lax-Milgram Theorem, and ultimately prove Theorem \ref{main1}. Finally, in Section \ref{sec:scatteredestimates} we deliver precise estimates for the scattered fields in terms of the incident fields. These estimates are known \cite{NWGK2014a}, but here we provide explicit dependence on the Lipschitz character of the domain in the resulting constants.

\section{Lipschitz Domains} \label{sec:Lipdomain}
We begin by letting $\Omega \subset \mathbb{R}^3$ be a Special Lipschitz domain, so that
\begin{align} \label{specialLipdomain}
\Omega = \left\{ (x', x_3): x_3 > \phi(x')\right\},
\end{align}
where $x'=(x_1, x_2)$ and $\phi: \mathbb{R}^2 \to \mathbb{R}$ is globally Lipschitz. Let $M$ denote the Lipschitz constant of $\phi$, i.e. 
\[ |\phi(x)-\phi(y)| \leq M |x-y|, \quad \forall \; x, y \in \mathbb{R}^2.\]
We have that boundary of $\Omega$ is 
$$\partial \Omega = \left\{ (x', x_3): x_3 = \phi(x')\right\},$$
and the unit normal $\nu$ is given by
$$\nu =\nu (x', \phi(x')) = \dfrac{1}{\sqrt{ 1+|\nabla_h \phi(x')|^2}} (\nabla_h \phi(x'), -1),$$
where the gradient $\nabla_h$ is the 2 dimensional ``horizontal" gradient. Following \cite{M2015traces} we let $\theta \in [0, \pi/2)$ be the angle defined by
\begin{align} \label{thetadef}
\theta = \arccos \left( \inf_{x' \in \mathbb{R}^2} \dfrac{1}{\sqrt{ 1+|\nabla_h \phi(x')|^2}} \right)
\end{align}
so that for $e=(0,0,1)$ we have
\begin{align}\label{verticaldirectioncondition}
-e \cdot \nu (x', \phi(x')) = \dfrac{1}{\sqrt{1+|\nabla_h \phi(x')|^2}} \geq \cos(\theta) \quad \forall \; x' \in \mathbb{R}^2.
\end{align}

\begin{example}\label{ex:wedge}
The wedge with angle $\alpha\in (0, \pi)$ can be defined as the Special Lipschitz domain
$$\Omega_{\alpha} \coloneqq \left\{(x,y)\in \mathbb{R}^2: y>\cot(\alpha/2)|x|\right\}; $$

\begin{figure}
\begin{center}
\includegraphics[scale=1.4]{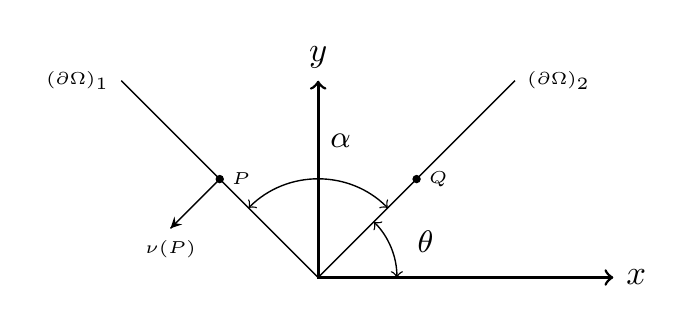}
\caption{The wedge domain $\Omega_{\alpha}$.}
\label{fig:SLD}
\end{center}
\end{figure}
\noindent see Figure \ref{fig:SLD}. The boundary of this domain is defined by the function $\phi(x)=\cot(\alpha/2)|x|$; the Lipschitz constant for this domain is $M=\cot(\alpha/2)$. In particular, for small angles $\alpha$, we see that
\[ M = M(\alpha) \approx \dfrac{2}{\alpha}. \]
Notice that for $\Omega_{\alpha}$, there holds
$$\dfrac{1}{\sqrt{ 1+|\phi'(x)|^2}}=\sin(\alpha/2)$$
so 
\[\theta= \arccos(\sin(\alpha/2)) = \dfrac{\pi}{2} - \dfrac{\alpha}{2}.
\] 
Let $(\partial \Omega)_1$ and $(\partial \Omega)_2$ denote the part of the boundary lying to the left and right of the angle $\alpha$, respectively. Then each of these pieces of $\partial \Omega$ can be parametrized:
\begin{align*}
(\partial \Omega)_1 &= \left\{ (-t\sin(\alpha/2), t\cos(\alpha/2): t\in \mathbb{R}_+\right\}, \\ (\partial \Omega)_2 &= \left\{ (t\sin(\alpha/2), t\cos(\alpha/2): t\in \mathbb{R}_+\right\}.
\end{align*}
Let $P \in (\partial \Omega)_1$ and $Q\in (\partial \Omega)_2$, so that
\[P=(-t\sin(\alpha/2), t\cos(\alpha/2))\quad \text{ and } \quad Q=(t\sin(\alpha/2), t\cos(\alpha/2)).\]
Then, simple geometric considerations imply that the normals at these points are given by
\[ \nu(P)= (-\cos(\alpha/2), -\sin(\alpha/2)) \quad \text{ and } \quad \nu(Q) = (\cos(\alpha/2), -\sin(\alpha/2))\]
so we see that the analog of (\ref{verticaldirectioncondition}) holds on both $(\partial \Omega)_1$ and $(\partial \Omega)_2$.
\end{example}

\section{Traces on Lipschitz Domains} \label{sec:traces}

The following gives precise control of the constant for the trace of an $H^1(\Omega)$ function on Special Lipschitz domains. 

\begin{theorem}[\cite{M2015traces}] \label{SMTheorem1}
Let $\Omega$ be a domain of the form (\ref{specialLipdomain}). Suppose $u \in H^1(\Omega)$. Then the trace of $u$, denoted $\left. u \right|_{\partial \Omega}$, belongs to $L^2(\partial \Omega)$ and
\begin{align} \label{traceestimate1}
||\left. u \right|_{\partial \Omega} ||^2_{L^2(\partial \Omega)} \leq \dfrac{2}{\cos(\theta)} ||u||_{L^2(\Omega)} ||\nabla u||_{L^2(\Omega; \mathbb{R}^3)},
\end{align}
where $\theta$ is defined in (\ref{thetadef}).
\end{theorem}

The analog of Theorem \ref{SMTheorem1} for a bounded, general Lipschitz domain is the following:

\begin{theorem}[\cite{M2015traces}] \label{SMTheorem2}
Let $\Omega$ be a bounded, Lipschitz domain. There exists a constant $C > 0$ such that for all $u \in H^1(\Omega)$, the trace of $u$ lies in $L^2(\partial \Omega)$ and
\begin{align} \label{traceestimate2}
\left|\left| \left. u \right|_{\partial \Omega} \right|\right|_{L^2(\partial \Omega)}^2 \leq \dfrac{1}{\beta} ||u||_{L^2(\Omega)} \cdot \left( 2 ||\nabla u||_{L^2(\Omega; \mathbb{R}^3)} + C ||u||_{L^2(\Omega)} \right),
\end{align}
where $\beta = \min_k \cos(\theta_k)$. 

\end{theorem}

The angles $\theta_k$ are defined via a partition of unity below; see (\ref{LD}). The following is also well known. Here we use the notation $\gamma u(x) = \left. u \right|_{\partial \Omega}(x)$ for $x \in \partial \Omega$. 

\begin{theorem}[\cite{D1996traces}]
Let $\Omega \subset \mathbb{R}^3$ be a domain of the form (\ref{specialLipdomain}), and suppose $1/2 < s \leq 1$. Then the trace operator $\gamma: H^s(\Omega) \to H^{s-1/2}(\partial \Omega)$ is continuous. 
\end{theorem}
It will be useful later for us to know the precise constants that arise in the previous result. Let us also take $s=1$, and fix $\Omega \subset \mathbb{R}^3$ to be a Special Lipschitz domain. To track the constants, let $\phi$ be as in (\ref{specialLipdomain}) and consider the operators
 \begin{align} \label{Tphidef}
 T_\phi: L^2(\Omega) \to L^2 (\mathbb{R}^3_+)\\ 
 T_{\phi}u (x) = u(x', x_3 +\phi(x')) \nonumber
 \end{align}
 and 
 \begin{align} \label{Sphidef}
 S_{\phi}: L^2 (\partial \Omega) \to L^2 (\mathbb{R}^2) \\
 S_{\phi} u(x') = u(x', \phi(x'))\nonumber
 \end{align}
 where $S_{\phi}$ is defined for a.e. $x' \in \mathbb{R}^2$. It was shown in \cite{D1996traces} that
$$||S_{\phi}^{-1}w ||_{H^{1/2}(\partial \Omega)} \leq \sqrt{1+2M^2} ||w||_{H^{1/2}(\mathbb{R}^2)}$$
for any $w \in H^{1/2} (\mathbb{R}^2)$. In fact, we can write this constant in terms of the angle $\theta$ in (\ref{thetadef}) by using the fact that the surface measure $d\sigma$ satisfies
$$\mathrm{d}\sigma = \sqrt{1+ |\nabla_h \phi(x')|^2} \mathrm{d}x' \leq \dfrac{1}{\cos(\theta)} \mathrm{d}x'$$by definition of $\theta$. Pushing this $\theta$ through the proof of Lemma 3 in \cite{D1996traces} gives the following estimate:
\begin{align} \label{thetaestimate1}
||S_{\phi}^{-1} w ||_{H^{1/2}(\partial \Omega)} \leq\dfrac{1}{\cos(\theta)} ||w||_{H^{1/2}(\mathbb{R}^2)} = \dfrac{1}{\cos(\theta)} ||w||_{H^{1/2}(\mathbb{R}^2)}  
\end{align}
for each $w \in H^{1/2}(\mathbb{R}^2)$. Note that in the case of the wedge with angle $\alpha$ as in Example \ref{ex:wedge}, this becomes
\[||S_{\phi}^{-1} w ||_{H^{1/2}(\partial \Omega)} \leq \dfrac{1}{\sin(\alpha/2)} ||w||_{H^{1/2}(\mathbb{R}^2)}.  \]

\noindent Now, we also need to control the trace of $H^{1} (\mathbb{R}^3)$ functions, and for this we appeal to \cite[Lemma 16.1]{Tartar2007} which gives the estimate

\begin{align} \label{Tartarbookestimate}
 ||\gamma_0 w ||_{H^{1/2}(\mathbb{R}^2)} \leq \pi ||w||_{H^1(\mathbb{R}^3)}
 \end{align}
 for each\footnote{The constant $\pi$ comes from taking $s=1$ in the general constant $C(s) = \int_{-\infty}^{\infty} \dfrac{\mathrm{d}t}{(1+t^2)^s}$.} $w \in H^1(\mathbb{R}^3)$.  The notation $\gamma_0$ stands for the trace $\gamma_0: H^1 (\mathbb{R}^n) \to H^{1/2} (\mathbb{R}^{n-1})$. 
 
Additionally, one can easily verify the relationship between the maps $S_\phi$ and $T_\phi$ is given via $\gamma_0$: namely, $S_\phi \gamma = \gamma_0 T_\phi$.
 Bringing all this together so far yields the estimate for $u \in H^{1}(\Omega)$: 

\begin{align}
 ||\gamma u||_{H^{1/2}(\partial \Omega)} &= || S_{\phi}^{-1} (S_{\phi} \gamma u)||_{H^{1/2}(\partial \Omega)} \nonumber \\ &\leq\dfrac{1}{\cos(\theta)} ||S_{\phi} \gamma u ||_{H^{1/2}(\partial \Omega)} \\ &=\dfrac{1}{\cos(\theta)} ||\gamma_0 T_\phi u||_{H^{1/2}(\mathbb{R}^2)} \nonumber \\ &\leq \dfrac{\pi}{\cos(\theta)} ||T_{\phi} u||_{H^1(\mathbb{R}^3_+)} . \nonumber
 \end{align}
It remains to estimate $||T_{\phi}u||_{H^1(\mathbb{R}^3_+)}$ in terms of the $H^1 (\Omega)$ norm of $u$.

Since $\partial_j \phi \leq M$ for $j=1,2$, we can readily see that
 \begin{align} \label{H1estimateforTphi}
 ||T_{\phi}u||_{H^1 (\mathbb{R}^3_+)}\leq \sqrt{1+2M^2}  ||u||_{H^1 (\Omega)}
 \end{align}
so that
\begin{align}
||\gamma u||_{H^{1/2}(\partial \Omega)} \leq \dfrac{\pi}{\cos(\theta)} \cdot \sqrt{1+2M^2} ||u||_{H^1(\Omega)},
\end{align}
as required.

Suppose now $\Omega$ is a general bounded Lipschitz domain. Recall that for a general Lipschitz domain, we can pass to a partition of unity. Precisely, given $\Omega \subset \mathbb{R}^3$ a bounded Lipschitz domain, we can find $N \in \mathbb{N}$, a partition of unity $\{ \eta_k \}_k$, and domains $\Omega_k$ such that:
\begin{align} \label{LD}
\begin{split}
\textbf{LD}: \left\{\begin{tabular}{l}
$\overline \Omega \cap \left( \cup_{k}^N \Omega_k\right) = \overline \Omega$;\\[0.5em]
$\text{supp}(\eta_k) \subset \Omega_k$ for each $1\leq k\leq N$;\\[0.5em]
$0\leq \eta_k \leq 1$ for each $1\leq k \leq N$; and \\[0.5em]
$\sum_{k=1}^N \eta_k(x)^2 = 1$ for all $x \in \Omega$ \\
\end{tabular}\right. \\
\end{split}
\end{align}

This can be done in such a way so that for $1\leq k \leq N$, there exists a direction $e_k$ and an angle $\theta_k \in [0, \pi/2)$ such that $-e_k \cdot \nu(x) \geq \cos(\theta_k)$ for all $x \in \partial \Omega \cap \Omega_k$, see \cite[Section 3.2]{M2015traces}. 
We have:
\begin{theorem} \label{Lipestimatethm}
Let $\Omega$ be a bounded, Lipschitz domain. Then, there exists a constant $C>0$ such that for all $u\in H^1(\Omega)$, the trace of $u$ lies in  $H^{1/2}(\partial \Omega)$ and
\begin{align} \label{Lipestimatethmestimate}
||\gamma u||_{H^{1/2}(\partial \Omega)} \leq C ||u||_{H^1(\Omega)}
\end{align}
where $C$ depends just on the geometry of $\Omega$. Precisely, 
\[C= \pi \sqrt{1+2M^2} \cdot \max\limits_{1\leq k \leq N} \dfrac{1}{\cos(\theta_k)}\]

\end{theorem}

\begin{proof}
Let $\Omega_k$ and $\eta_k$ be as in (\ref{LD}). Then, for any $u\in H^1(\Omega)$, we have
$$\eta_k u \in H^1 (\Omega \cap \Omega_k), \quad 1\leq k \leq N$$
and
$$\gamma u (x) =\sum_{k=1}^N (\eta_k u)(x), \quad x\in \partial \Omega .$$
Thus we see that

\begin{align*}
||\gamma u||_{H^{1/2}(\partial \Omega)} = \left|\left| \sum_{k=1}^N \eta_k u \right|\right|_{H^{1/2}(\partial \Omega)} &\leq \sum_{k=1}^N ||\eta_k u ||_{H^{1/2}(\partial ( \Omega \cap \Omega_k))} \\ &\leq \sum_{k=1}^N \dfrac{\pi}{\cos(\theta_k)}\cdot \sqrt{1+2M^2} ||\eta_k u||_{H^1(\Omega \cap \Omega_k)} \\ 
& \leq \pi L \sqrt{1+2M^2} ||u||_{H^1(\Omega)}
\end{align*}
where $L = \max\limits_{1\leq k \leq N} \dfrac{1}{\cos(\theta_k)}$. 
This proves (\ref{Lipestimatethmestimate}) with $C= \pi L \sqrt{1+2M^2}$.
\end{proof}

\subsection{Other traces}
Following \cite{buffa2002}, we introduce the following notation. Denote by $\nu$ the unit normal to $\partial \Omega$, and let $\gamma_t (\u) = \left. \nu \times \u \right|_{\partial \Omega}$ and $\pi (\u) = \left. \nu \times (\u \times \nu) \right|_{\partial \Omega}$; we will make precise where these trace operators act shortly. A priori they are defined from the space of distributions on $\overline{\Omega}$ to $L^2_t(\partial \Omega)$; see \eqref{eq_L_t_space} in the Appendix for a definition of $L^2_t(\partial \Omega)$. 
 
Denote by $\widetilde{\gamma_t}$ and $\widetilde{\pi}$ the composition operators $\gamma_t \circ \gamma^{-1}$ and $\pi \circ \gamma^{-1}$, respectively, where $\gamma: \left( H^1(\Omega)\right)^3 \to \left(H^{1/2}(\partial\Omega)\right)^3$ is the standard vector trace operator. Let $V_{\gamma} = \widetilde{\gamma_t} \left( H^{1/2}(\partial \Omega)\right)^3$ and $V_{\pi} = \widetilde{\pi} \left( H^{1/2}(\partial \Omega)\right)^3$. When endowed with the following norms, these spaces become Hilbert spaces:
$$||\lambda ||_{V_{\gamma}} = \inf_{\u \in \left(H^{1/2}(\partial \Omega)\right)^3}\left\{ ||\u||_{\left(H^{1/2}(\partial \Omega)\right)^3} \;: \widetilde{\gamma_t} (\u) = \lambda \right\},$$
$$||\lambda ||_{V_{\pi}} = \inf_{\u \in \left(H^{1/2}(\partial \Omega)\right)^3}\left\{ ||\u||_{\left(H^{1/2}(\partial \Omega)\right)^3} \;: \widetilde{\pi} (\u) = \lambda \right\}.$$
Note that, if $\Omega$ was smooth, the spaces $V_{\pi} =V_{\gamma} = TH^{1/2}(\partial \Omega)$, the standard space of tangential vector fields of order $1/2$ on $\partial \Omega$; see (\ref{THonehalf}) in the Appendix. 
The following Green's formula holds (see \cite[Equation (27)]{buffa2002}) for all $\u \in \textbf{H}(\text{curl},\Omega)$ and $\v \in \left(H^1(\Omega)\right)^3$:

\begin{align} \label{supergeneralGreen}
\int_{\Omega} \left( \u \nabla \times \v - \v \nabla \times \u \right) \; \mathrm{d}x = -_{\gamma'}\langle \pi(\u), \gamma_t (\v) \rangle_{\gamma}
\end{align}
where $_{\gamma'}\langle \cdot, \cdot \rangle_{\gamma}$ denotes the duality pairing between $V_{\gamma}$ and $V_{\gamma}'$. 

\subsubsection{Estimates of the trace operator \texorpdfstring{$\gamma_t$}{gammat}}

We recall the following useful result, where $\T$ is defined by \eqref{eq:def_T-trace_space} in the Appendix,

\begin{theorem}[\cite{Alonso1996}, Lemma 2.2] \label{AlonsoTheorem}
Suppose $\Omega \subset \mathbb{R}^3$ is a bounded, Lipschitz domain. Then for any $\u \in \Hcurl$, 
\begin{align} \label{alonsoestimate}
\left. ||\gamma_t (\u) \right|_{\partial \Omega} ||_{\T} \leq C_{\Omega} ||\u||_{\Hcurl},
\end{align}
where $C_\Omega = (\sqrt{2}+1)C_1$, with $C_1$  given by the norm of any continuous extension operator $E_1: H^{1/2}(\partial \Omega) \to H^1(\Omega)$.
\end{theorem}

We next proceed to uncover the precise constant $C_1$ above in terms of the geometry of the Lipschitz domain.

To this end we start with a quantitative version of Theorem 3.23 in \cite{Mclean}, which we will need in the Sobolev space $H^{1/2}$. 

\begin{theorem} \label{quantitativeCOV}
Suppose $\phi: \Omega_1 \to \Omega_2$ is bi-Lipschitz between two sets $\Omega_1, \Omega_2 \subset \mathbb{R}^n$ with constant $M$, where $M$ is the largest of the Lipschitz constants of $\phi$ and $\phi^{-1}$. For $0 \leq s\leq 1$, we have $u \in H^s (\Omega_2)$ if and only if $u\circ \phi \in H^s (\Omega_1)$. 
\end{theorem}

\begin{proof}
We prove this by interpolation. For $s=0$ we identify $H^0 (\Omega_j)$ with $L^2 (\Omega_j)$ and appeal to the change of variables formula
$$\int_{\Omega} (f\circ T) |J_{T}| = \int_{T(\Omega)} f$$
where $T: \Omega \to \mathbb{R}^n$ is bi-Lipschitz with Jacobian $J_{T}$. 

If $J_{\phi}$ denotes the Jacobian of $\phi$, then since $\phi$ is bi-Lipschitz\footnote{Note that by Rademacher's Theorem, a bi-Lipschitz function is differentiable a.e.}, we have that $|J_{\phi}| > M^{-1}$.
Via the change of variables formula this in particular implies that $||u \circ \phi||_{L^2(\Omega_1)}^2 \leq M ||u||_{L^2(\Omega_2)}^2$. Similarly, we have that
$$|D (u\circ \phi)|^2 \leq M |D u(\phi(x))|^2 \cdot |J_{\phi}(x)|$$ so that 
$$||D(u\circ \phi)||_{L^2(\Omega_1)}^2 \leq M||u||_{L^2(\Omega_2)}^2 .$$
 
Together these estimate imply that
$$||u \circ \phi||_{H^1(\Omega_1)}^2 \leq M ||u||_{H^1(\Omega_2)}^2 .$$
Since the $s=0$ and $s=1$ cases are verified, we interpolate to obtain the range $0<s<1$. By the interpolation theorem (Theorem B.8 in \cite{Mclean}), we know that the $H^s$ norm of $u$ will equal the following so-called $K_{\zeta,2}$ norm 
\begin{align*}
||u||_{ K_{\zeta,2}} = N_{\zeta,2} ||K(\cdot, u)||_{\zeta, 2}
\end{align*}
where
$$N_{\zeta,2} = \left( \dfrac{2\sin (\pi \zeta)}{\pi}\right)^{1/2}$$ and
$$||f||_{\zeta, 2}= \left( \int_0^{\infty} |t^{-\zeta} f(t)|^2 \dfrac{\mathrm{d}t}{t} \right)^{1/2}$$
for $0<\zeta <1$. Recall that for a compatible pair of Banach Spaces $(X_0, X_1)$\footnote{Here the spaces are $L^2(\Omega)$ and $H^1(\Omega)$.},
the $K$-functional is defined by
\[K(t,u) =\inf\left\{ \left( ||u_0||_{X_0}^2+t^2 ||u_1||_{X_1}^2 \right)^{1/2}: u_0 \in X_0, \; u_1 \in X_1, \; \text{ and } \; u_0+u_1=u \right\}. 
\]
Thus we have for $0<s<1$,
\begin{align}\label{eq:interpol_inequality}
||u\circ\phi ||_{H^s(\Omega_1)} = ||u\circ\phi ||_{K_{\zeta, 2}} \leq M^{1-\zeta} \cdot M^{\zeta} ||u||_{K_{\zeta,2}} =M ||u||_{H^s(\Omega_2)}
\end{align}
where in the inequality we have used the interpolation inequality. Notice that we indeed have equality between the Sobolev norm and the $K_{\zeta,2}$ norm by taking the previous normalization $N_{\zeta, 2}$. 

\end{proof}

\begin{proposition}\label{prop:extension}
Let $\Omega \subset \mathbb{R}^3$ be a bounded, Lipschitz domain. For each $u \in H^{1/2}(\partial \Omega)$, there exists a continuous extension of $u$ to $H^1(\Omega)$, denoted $Eu$, such that 
\begin{align} \label{extensionestimate}
||Eu||_{H^1(\Omega)} \leq M k_1 \; ||u||_{H^{1/2}(\partial \Omega)}
\end{align}
for some constant $k_1$.
\end{proposition}

\begin{proof}
We construct the extension operator $E: H^{1/2}(\partial \Omega) \to H^1 (\Omega)$ as in Theorem \ref{AlonsoTheorem}, where we track all the constants involved. To this end we begin by locally flattening the boundary $\partial \Omega$. Let $x_0 \in \partial \Omega$ and let $B(x_0, R)$ denote a ball of radius $R$ centered at $x_0$. Since $\Omega$ is Lipschitz, we can find a bi-Lipschitz function $\Phi: B(x_0, R)\to U \subset \mathbb{R}^3$ such that
$$\Phi \left( B(x_0, R) \cap \partial \Omega\right) = U \cap \{ x_3 =0\} .$$
Consider the operator $\eta : H^{1/2}(\mathbb{R}^2)\to H^1 (\mathbb{R}^3)$ given by 
\begin{align} \label{etadef}
\eta u (x) = \int_{\mathbb{R}^2} \widehat u(\xi) f \left[ (1+|\xi|^2)^{1/2} x_3\right] e^{2\pi i \xi\cdot x}\; \mathrm{d}\xi, \quad x \in \mathbb{R}^3
\end{align} 
where $f \in \mathcal{D}(\mathbb{R})$ satisfies

$f(y) = 1$ for $|y|\leq 1$, and $\widehat u$ denotes the Fourier transform of $u$. Then $\eta$ is a bounded linear operator by Lemma 3.36 in \cite{Mclean}. Let $u \in H^{1/2}(\partial \Omega)$; we will extend $u$ to $H^1(\Omega)$ as follows. Near $x_0 \in \partial \Omega$ define $\widetilde u = \eta(u \circ \Phi)$. We first show that $\widetilde u \in H^1 (U)$ where $U$ is the neighborhood as above. We see that
\begin{align*}
||\widetilde u ||_{H^1(U)} = ||\eta (u \circ \Phi)||_{H^1(U)} = k_1 ||u \circ \Phi||_{H^{1/2}(U \cap \{ x_3=0\})} \leq M\; k_1 \; ||u||_{H^{1/2}(B(x_0, R) \cap \partial \Omega)}
\end{align*}
where the constant $k_1$ is given by
\begin{align} \label{k1def}
k_1 = \int_{\mathbb{R}} (1+t^2) |\widehat f (t)|^2 \; \mathrm{d}t
\end{align}
and we have used Theorem \ref{quantitativeCOV}, in particular \eqref{eq:interpol_inequality}. 
Thus locally we have extended $u$. Next we piece this together via a partition of unity. Let $\{\eta_j\}$ be such a partition of unity subordinate to the open cover $\{\Omega_j\}$ of $\Omega$. Then define $E\coloneqq \sum_{j=1}^N \eta_j E_j$ where the $E_j$'s are extensions on each $\Omega_j$, constructed as above. Then we have
$$||Eu||_{H^1(\Omega)} \leq \sum_{j=1}^N M\; k_1 ||u||_{H^{1/2}(\partial \Omega_j)} \leq M \; k_1 \; ||u||_{H^{1/2}(\partial \Omega)}$$
as required. 
\end{proof}

\begin{remark}
To construct the function $f$, one can proceed via a standard argument. Define
\[g_1(x) = \begin{dcases}{\exp(-1/(x+2))} &\mbox{ if } \; x\geq -2\\
0 &\mbox{ if } \; x\leq -2
\end{dcases}.
\]
Then let $g_2(x)\coloneqq g_1(1+x)g_1(-3-x)$.
The function $g_2(x)$ is smooth and supported in $(-3,-1)$. Define
\[g_3(x) \coloneqq \int_{-\infty}^x g_2(s)\; \mathrm{d}s .\]
Then $g_3(x)=0$ for $x<-3$ and is constant (we'll call this constant $C_g$--which is approximately $0.133086$) for $x>-1$. Finally, let $g_4(x) = g_3(x) g_3(-x)$. Then $g_4(x)=0$ for $|x|>3$ and is constant ($C_g$) for $|x|<1$; so we can take $f(x) = g_4(x)/C_g$. Notice also that since $f$ is compactly supported, by the Paley-Wiener Theorem (\cite[Theorem 7.2.2]{Strichartz2003}) $\widehat{f}(t)$ is an entire function of exponential type and satisfies the bound
\[|\widehat{f}(t)|\leq C_N (1+|t|)^{-N} \quad \forall \; N >1,
\]
so $k_1$ is indeed finite.
\end{remark}

\subsubsection{Estimates for the trace operator $\pi$}

In this section we develop a similar estimate for the trace operator $\pi$. We begin with a proposition.

\begin{proposition} \label{piprop}
There exists a constant $C_{\pi}^1>0$ such that for any $\u \in \textbf{H}(\text{curl},\Omega)$, there holds
\begin{align} \label{firstpiestimate}
||\pi (\u)||_{V_{\gamma}'} \leq C_{\pi}^1 ||\u||_{\textbf{H}(\text{curl},\Omega)} .
\end{align}

\end{proposition}
\noindent We delay the proof of Proposition \ref{piprop} until after the proof of Theorem \ref{piestimate2} below as it uses a similar argument. Our main estimate for the operator $\pi$ is:

\begin{theorem} \label{piestimate2}
There exists a constant $C_{\pi}^2>0$ depending just on the geometry of $\Omega$ so that for any $\u \in \textbf{H}(\mathrm{curl},\Omega)$, there holds
\begin{align} \label{mainpiestimate}
\begin{split}
||\mathrm{curl}_{\partial \Omega}  (\pi (\u))||_{H^{-1/2}(\partial \Omega)} \leq C_{\pi}^2 ||\u||_{\textbf{H}(\mathrm{curl},\Omega)} . \\
\end{split}
\end{align}
Precisely, $C_{\pi}^2 = M k_1$.
\end{theorem}

\noindent The map $\text{curl}_{\partial \Omega}$ is defined by duality as follows: $\text{curl}_{\partial \Omega}: V_{\gamma}' \to H^{-3/2}(\Omega)$ is given by
\begin{align} \label{dualityofcurl}
\langle \text{curl}_{\partial \Omega} (\lambda), \phi \rangle_{3/2} = _{{\gamma}'}\langle \lambda, \nabla_{\partial \Omega} \times \phi \rangle_{\gamma} \qquad \lambda \in V_{\gamma}' \quad \phi \in H^2(\Omega)
\end{align}
where $\langle \cdot, \cdot \rangle_{3/2}$ denotes duality between $H^{3/2}(\partial \Omega)$ and $H^{-3/2}(\partial \Omega)$ and $\nabla_{\partial \Omega} \times: H^1(\partial \Omega) \to L^2_t(\partial \Omega)$ is the surface curl (see the Appendix). Note that from \cite[Proposition 3.4]{buffa2002}, the operator $\nabla_{\partial\Omega}\times: H^{3/2}(\partial\Omega) \to V_{\gamma}$ is linear and continuous.

With these estimates in hand one can finally realize the map 
\[\pi: \textbf{H}(\text{curl},\Omega) \to H^{-1/2}(\text{curl}_{\partial \Omega}, \partial \Omega)\] with an appropriate trace estimate, where the space $H^{-1/2}(\text{curl}_{\partial\Omega},\partial \Omega)$ is defined by
\begin{align} \label{tracespaceforpi}
H^{-1/2}(\text{curl}_{\partial\Omega}, \partial \Omega) = \left\{ \lambda \in V_{\gamma}' : \text{curl}_{\partial \Omega}( \lambda) \in H^{-1/2}(\partial \Omega) \right\}.
\end{align}

\begin{proof}[Proof of Theorem \ref{piestimate2}]
Let $\u \in \textbf{H}(\text{curl},\Omega)$. By Proposition \ref{piprop}, we have that $\pi (\u) \in V_{\gamma}'$. Now we need to check that $\text{curl}_{\partial \Omega} (\pi (\u))$ belongs to $H^{-1/2}(\partial \Omega)$. Let $\phi \in H^2(\Omega)$. Then we have
\begin{align*}
\langle \text{curl}_{\partial \Omega} (\pi (\u)), \phi \rangle_{3/2} = \; _{{\gamma}'} \langle \pi(\u), \nabla_{\partial \Omega} \times \phi \rangle_{\gamma} =\;  _{{\gamma}'} \langle \pi (\u), \gamma_t (\nabla \phi) \rangle_{\gamma} = -\int_{\Omega} \nabla \phi \cdot \nabla \times \u \; \mathrm{d}x
\end{align*}
where the first equality follows from (\ref{dualityofcurl}), the second equality follows from \cite[Proposition 3.4]{buffa2002}, and the last equality follows from (\ref{supergeneralGreen}). 
Now, since $\phi \in H^2(\Omega)$, we have $\left. \nabla \phi \right|_{\partial \Omega} \in \left(H^{1/2}(\partial \Omega)\right)^3$. Denote by $V$ the extension of $\nabla \phi$ to $\Omega$; then as in Proposition \ref{prop:extension}, we have that $||V||_{\left(H^1(\Omega)\right)^3} \leq M \; k_1 ||\nabla \phi||_{\left(H^{1/2}(\partial \Omega)\right)^3}$. We then obtain 
\begin{align*}
\langle \text{curl}_{\partial \Omega} (\pi (\u)), \phi \rangle_{3/2} = _{{\gamma}'} \langle \pi (\u), \gamma_t (\nabla \phi) \rangle_{\gamma} &= -\int_{\Omega} V \cdot \nabla \times \u \; \mathrm{d}x \\ &\leq ||V||_{\left(H^1(\Omega)\right)^3} ||\u||_{\Hcurl} \\ &\leq M\; k_1 ||\u ||_{\Hcurl} ||\nabla \phi||_{\left(H^{1/2}(\partial \Omega)\right)^3}.
\end{align*}

This in particular means that $\text{curl}_{\partial \Omega}(\pi (\u)) \in H^{-1/2}(\partial \Omega)$ and (\ref{mainpiestimate}) holds with constant $C_{\pi}^2 =Mk_1$ depending just on the geometry of $\Omega$.

\end{proof}
\noindent Next we move to the 
\begin{proof}[Proof of Proposition \ref{piprop}]
The proof goes along the lines of the proof of Theorem \ref{piestimate2}. We write
$$||\pi(\u)||_{V_{\gamma}'} = \sup_{g \in V_{\gamma}} \dfrac{_{\gamma'}\langle \pi(\u), g \rangle_{\gamma}}{||g||_{V_{\gamma}}}$$
and extend $g$ from $V_{\gamma}$ to $\left(H^1(\Omega)\right)^3$; denote its extension by $G$. Again using the estimate for the extension operator from Proposition \ref{prop:extension}, we obtain (\ref{firstpiestimate}) with $C_{\pi}^1 = M\; k_1$ since
\[ |_{\gamma'}\langle \pi(\u), g \rangle_{\gamma}| \leq ||\u ||_{\textbf{H}(\text{curl},\Omega)} \cdot ||G||_{\left(H^1 (\Omega)\right)^3} .\]
\end{proof}

\begin{remark}
An alternative way to prove Proposition \ref{piprop} is the following. Since $g \in V_{\gamma}$, there is some $\mu \in \left(H^{1/2}(\partial \Omega)\right)^3$ so that $||g||_{V_{\gamma}} = ||\mu||_{\left(H^{1/2}(\partial \Omega)\right)^3}$. Now, choose $P$ 
to be the weak solution of $-\Delta P + P = 0$ in $\Omega$ with boundary condition $P = \mu$ on $\partial \Omega$. Then by elliptic regularity (e.g. \cite{Mclean}, Theorem 4.10), there exists $C>0$ so that $||P||_{\left(H^1(\Omega)\right)^3} \leq C ||\mu||_{H^{1/2}(\partial \Omega)} = C ||g||_{V_{\gamma}}$. Combining all of this together gives the estimate (\ref{firstpiestimate}).
\end{remark}

\section{The scattering problem} \label{sec:scattering}

\subsection{The Incident Fields}
Let $\Omega \subset \mathbb{R}^3$ be a Lipschitz domain. Let $\nu \in L^\infty (\partial \Omega)$ denote the outer unit normal. Define the exterior of the domain $\Omega$ to be $\Omega_e \coloneqq \mathbb{R}^3 \setminus \overline \Omega$ which we assume to be vacuous. Incident fields $\E_i$ and $\B_i$ originate in a region $\Omega_i$ such that $\Omega \cap \Omega_i = \emptyset$. Outside of this region we assume the fields satisfy the time-harmonic Maxwell system

\begin{align} \label{timeharmonicMaxwell}
\begin{dcases}
\nabla \times \E_i (\x) = i k_0  \B_i (\x),\\
\nabla \times \B_i (\x) =-i k_0 \E_i(\x)\\
\end{dcases}
\end{align}
for $\x \in \mathbb{R}^3$. Here $k_0 = \omega/c$ is the wavenumber in vacuum with $\omega$ the angular frequency of the fields. We assume the incident fields have tangential traces on $\partial \Omega$ belonging to $\textbf{T}(\partial \Omega)$, i.e.
$$(\gamma_t( \E_i), \gamma_t( \B_i)) \in \textbf{T}(\partial \Omega) \times \textbf{T}(\partial \Omega) ,
$$
where $\textbf{T}(\partial \Omega)$ is defined in the Appendix.

\begin{remark}
In the case that $\Omega$ is actually smooth, the space $\textbf{T}(\partial \Omega)$ becomes the classical space $H^{-1/2}(\text{div}, \partial \Omega)$, see (\ref{Hminusonehalfdivdef}). This is consistent with the problem setup in e.g. \cite{NWGK2003}. Furthermore, the definition of $\textbf{T}(\partial \Omega)$ does not require the domain to be of the form (\ref{specialLipdomain}). 

\end{remark}

\subsection{On the Dual of \texorpdfstring{$\T$}{TOmega}} \label{sec:dualofT}
In the case that $\Omega$ is actually smooth, we saw above that $\T = H^{-1/2}(\text{div}, \partial \Omega)$. It is well known that the dual of this space is the space
$$H^{-1/2}(\text{curl},\partial \Omega) = \left\{ U \in H_t^{-1/2}(\partial \Omega): \nabla_{\partial \Omega} \times U \in H^{-1/2}(\partial \Omega)\right\} .$$

In \cite{buffa2002}, it is shown that $\T$ is isomorphic to the space 
\[H^{-1/2}(\text{div}_{\partial \Omega}, \partial \Omega)=\left\{ \u \in V_{\pi}^{'}: \text{div}_{\partial \Omega} \u  \in H^{-1/2}(\partial \Omega) \right\}\]
where the surface divergence is defined by $\text{div}_{\partial \Omega}: V_{\pi}' \to H^{-3/2}(\partial \Omega)$ and is given by the formula
\[\langle \text{div}_{\partial \Omega} \mathbf{u}, \phi \rangle_{3/2} =-\langle \mathbf{u}, \nabla_{\partial \Omega} \phi \rangle_{V_{\pi}}. \quad \mathbf{u} \in V_\pi^{'}, \; \phi \in H^2(\Omega)\]
where $\langle \cdot, \cdot \rangle_{3/2}$ denotes the duality pairing between $H^{-3/2}(\partial \Omega)$ and $H^{3/2}(\partial \Omega)$, and $\langle \cdot, \cdot \rangle_{V_{\pi}}$ denotes the duality pairing between $V_\pi'$ and $V_\pi$. Further, the operator $\nabla_{\partial\Omega}$ is defined in the Appendix; see the comments directly after (\ref{THonehalf}). In fact, the isomorphism in \cite{buffa2002} is given explicitly by 
$$i_{\pi}: V_{\pi}'\to \left( \text{ker}(\pi) \cap H^{1/2}(\partial \Omega) \right)^\circ $$
where $\circ$ denotes the polar set (see e.g. the definition on page 136 in \cite{Yosida}). Intuitively, this isomorphism takes the same vector into itself with zero normal component .
Now, from Theorem 4.1 in \cite{buffa2002}, it follows that the dual (with pivot space $L^2_t (\partial \Omega)$) of $H^{-1/2}(\text{div}_{\partial \Omega}, \partial \Omega)$ is given by (\ref{tracespaceforpi}). Classical functional analysis says then (e.g. \cite{Megginson} Theorem 1.10.12) the map 
$$i_\pi^*: \T ' \to H^{-1/2}(\text{curl}_{\partial\Omega},\partial \Omega)$$
with the latter space given by (\ref{tracespaceforpi}), is an isomorphism.

\subsection{The Interior Problem}
We assume the material parameters $\epsilon, \mu \in L^\infty (\Omega; \mathbb{C}^{3\times 3})$ and satisfy the coercivity condition
$$\left(
\begin{array}{c}
a\\
b\\
\end{array}
\right)^{\dagger}\begin{pmatrix}
-i k_0 (\epsilon (\x) -\epsilon^{\dagger}(\x)) & 0 \\
0 & -i k_0 (\mu(\x)-\mu^{\dagger}(\x))\\
\end{pmatrix} \cdot
\left(
\begin{array}{c}
a\\
b\\
\end{array}
\right)
\geq c (|a|^2 + |b|^2)
$$
for almost every $x \in \Omega$ and for all $a, b\in \mathbb{C}^3$. Here $c>0$ is a constant that may depend on $\omega$. The dagger represents conjugate transpose, so that $(A^{\dagger})_{ij} = \overline A_{ji}$ for a $3\times 3$ matrix $A$. Note that physically this condition means that the material is lossy (almost everywhere). 

In $\Omega$ the fields $\E, \B$ satisfy the Maxwell equations
\begin{align} \label{timeharmonicMaxwell-interior}
\begin{dcases}
\nabla \times \E (\x) = i k_0\mu(\x)  \B (\x),\\
\nabla \times \B (\x) =-i k_0\epsilon(\x) \E(\x)\\
\end{dcases}
\end{align}
for $\x \in \Omega$. 

\subsection{The Exterior Problem}
The domain $\Omega$ distorts the incident fields, creating scattered fields which we denote by $\E_s, \B_s$, which belong to $H_{\text{loc}}(\text{curl},\overline \Omega_e)$ and satisfy
\begin{align} \label{timeharmonicMaxwell-exterior}
\begin{dcases}
\nabla \times \E_s (\x) = i k_0 \B_s (\x),\\
\nabla \times \B_s (\x) =-i k_0\E_s(\x)\\
\end{dcases}
\end{align}
for $x \in \overline \Omega_e$. The Silver-M\"uller radiation condition should also be satisfied by $\E_s$ or $\B_s$ (i.e. one of the following should be satisfied):

\begin{align} \label{Silver-Muller}
\hat \x \times \E_s (\x) - \B_s (\x) = o (1/x) \quad \text{ as } x \to \infty \\
\hat \x \times \B_s (\x) + \E_s(\x) = o(1/x) \quad \text{ as } x \to \infty
\end{align}
for all directions $\hat \x$.

\subsection{Boundary Conditions}
In $\Omega_e$, the total field is the sum of the incident and scattered fields:
\begin{align*}
\begin{dcases}
\E_t (\x) = \E_i (\x) + \E_s (\x) \\
\B_t (\x) = \B_i (\x) + \B_s (\x)\\
\end{dcases}
\end{align*}
for $\x \in \Omega_e$. The boundary conditions on $\partial \Omega$ are
\begin{align} \label{boundaryconditions}
\begin{dcases}
\gamma_t^+ (\E_i +\E_s) = \gamma_t^- (\E)\\
\gamma_t^+(\B_i + \B_s) = \gamma_t^- (\B)\\
\end{dcases}
\end{align}
where the plus and minus denote the trace from the interior or exterior, and $\gamma_t: \textbf{H}(\text{curl},\Omega) \to \textbf{T}(\partial \Omega)$ is the tangential trace map $\left. \E \mapsto \nu \times \E \right|_{\partial \Omega}$. 

\subsection{Calder\'on Operators}
The exterior Calder\'on operator $C^e$ is defined by
\begin{align} \label{exteriorCalderondef}
C^e : \T \to \T, \quad \gamma_t^+ (\E_s ) \mapsto \gamma_t^+ (\B_s)
\end{align}
where the fields $\E_s, \B_s$ solve the following exterior problem:

\begin{numcases}{}
(\E_s, \B_s) \in H_{\text{loc}}(\text{curl},\overline{\Omega}_e) \times H_{\text{loc}}(\text{curl},\overline{\Omega_e})\label{firsteqn}, \\ 
	\begin{dcases}
    \nabla \times \E_s(\x) = ik_0 \B_s(\x), \quad \x \in \Omega_e, \\ \label{secondeqn}
    \nabla \times \B_s (\x) = -ik_0 \E_s (\x), \quad \x \in \Omega_e, \label{thirdeqn}
    \end{dcases}\\
    \begin{dcases}
    \hat \x \times \E_s (\x) - \B_s (\x) = o (1/x) \; \text{or}\\
    \hat \x \times \B_s (\x) + \E_s (\x) = o (1/x)\; \text{as} \; x \to \infty, \\
    \end{dcases} \\
     \gamma_t^+ (\E_s) = \m \in \T \label{lasteqn}
\end{numcases}

In the case that $\Omega$ is smooth (e.g. has a $C^{1,1}$ boundary), solvability of (\ref{firsteqn})-(\ref{lasteqn}) is known and can be found e.g. in \cite{Cessenat}. In the case of Lipschitz domains, uniqueness of solutions of (\ref{firsteqn})-(\ref{lasteqn}) was shown in \cite{buffa2005}, see also \cite{buffa2002boundary}. This in particular implies that the Calder\'on operator $C^e$ is uniquely defined for all $m \in \T$. 

Some useful properties of the Calder\'on operator are collected here. In what follows, $\mathrm{d}\sigma$ denotes the surface measure on $\partial \Omega$.

\begin{theorem} \label{Calderonproperties}
The exterior Calder\'on operator defined by (\ref{exteriorCalderondef}) satisfies the following properties:

\begin{enumerate}
\item Positivity: for all $\m \in \T$, there holds
	\begin{align} \label{positivity}
	\Re \int_{\partial \Omega} C^e (\m) \cdot (\nu \times \overline \m) \; \mathrm{d}\sigma \geq 0
	\end{align}
\item Boundedness: 
	\begin{align*}
    \left( C^e \right)^2 = - \textbf{I} \quad \text{ on } \T
    \end{align*}
\item Isomorphism: the operator $C^e$ is an isomorphsim on $\T$, so there exists constants $0< c_C \leq C_C$ so that
	\begin{align}
    c_C ||\m||_{\T} \leq ||C^e (\m)||_{\T} \leq C_C ||\m||_{\T}
    \end{align}
\end{enumerate}
\end{theorem}
The proof of these follows directly from methods in \cite{Cessenat}. The third item in particular implies we can take
$$
C_C = ||C^e||_{\T} \quad \text{ and } \quad c_C = ||C^e||_{\T}^{-1}
$$
Recall that
\[ ||C^e||_{\T} \coloneqq \sup\limits_{||\v||_{\T}=1} ||C^e(\v)||_{\T}\]

Finally, as a result of (1) in Theorem \ref{Calderonproperties}, we have that
\begin{align} \label{positivity2}
- \Re \int_{\partial \Omega} \overline{\pi_- (\u)} \cdot C^e (\gamma_t^-(\u)) \; \mathrm{d}\sigma \geq 0, \quad \u \in \textbf{H}(\text{curl},\Omega)
\end{align}

\section{Weak Formulation and Solution} \label{sec:weakformulation}
We give here a weak formulation of the problem in terms of a sesquilinear form. 

\begin{definition}
For $\u, \v \in \textbf{H}(\text{curl},\Omega)$, define
\begin{equation} \label{sesquilinearformdef}
A(\u, \v) = \int_{\Omega} \left( \dfrac{i}{k_0} \overline{(\nabla \times \v)} \cdot \mu^{-1} \cdot (\nabla \times \u) -i k_0 \overline{\v}\cdot \epsilon \cdot \u \right) \mathrm{d}x - \int_{\partial \Omega} \overline{\pi_- (\v)} \cdot C^e (\gamma_t^- (\u))\; \mathrm{d}\sigma .
\end{equation}
\end{definition}

\noindent A weak formulation of the original problem then is to find $\E \in \textbf{H}(\text{curl},\Omega)$ such that

\begin{align} \label{weakformulation}
A(\E, \v) = \int_{\partial \Omega} \left( \gamma_t^+ (\B_i) -C^e (\gamma_t^+ (\E_i)) \right) \cdot \overline{\pi_-(\v)} \; \mathrm{d}\sigma \quad \forall \; \v \in \textbf{H}(\text{curl},\Omega).
\end{align}
The solution satisfies the boundary conditions
\begin{align*} 
\begin{dcases}
\gamma_t^+ (\E_i +\E_s) = \gamma_t^- (\E)\\
\gamma_t^+(\B_i + \B_s) = \gamma_t^- (\B)\\
\end{dcases}
\end{align*}
and couples to the solution of the exterior problem. The corresponding magnetic field is then constructed as 
\[ \textbf{H}(x) = \dfrac{-i}{k_0} \mu^{-1} \cdot \nabla \times \textbf{E}(x). \]

\subsection{Proof of Theorem \ref{main1}}

Using our trace estimates for $\gamma_t$ and $\pi$, we can see that $A(\u, \v)$ is indeed bounded:

\begin{align*}
\left| \int_{\partial \Omega}  \overline{\pi_- (\v)} \cdot C^e (\gamma_t^- (\u))\; \mathrm{d}\sigma \right|^2 &\leq ||C^e||_{\T}^2 ||\pi_-(\v)||_{\T '}^2 ||\gamma_t^- (\u)||_{\T}^2 \\ &\leq C \cdot ||C^e||_{\T}^2 ||i_{\pi}^\ast \left( \pi_-(\v) \right)||_{H^{-1/2}(\text{curl},\Omega)} ||\gamma_t^- (\u)||_{\T}^2 \\& \leq \max \left( (C_{\pi}^1)^2, (C_{\pi}^2)^2 \right)||\v||_{\textbf{H}(\text{curl},\Omega)}^2 ||C^e||_{\T}^2||\gamma_t^-(\u)||_{\T}^2 \\ &\leq \max \left( (C_{\pi}^1)^2, (C_{\pi}^2)^2 \right)C_{\Omega}^2 ||\v||_{\textbf{H}(\text{curl},\Omega)}^2  ||C^e||_{\T}^2 ||\u||_{\textbf{H}(\text{curl},\Omega)}^2 \\ &\coloneqq K^2 ||\v||_{\textbf{H}(\text{curl},\Omega)}^2 ||\u||_{\textbf{H}(\text{curl},\Omega)}^2 .
\end{align*}
The second inequality follows from the fact that $i_{\pi}^\ast$ is an isomorphism. Recalling\footnote{See Theorem \ref{AlonsoTheorem}, where the norm of the extension operator we have constructed is $M \; k_1$.}  that $C_{\Omega}=(1+\sqrt{2})M\; k_1$and $C_{\pi}^1=C_{\pi}^2=M\; k_1$, we obtain the constant
\[K^2=k_1^4 M^4(1+\sqrt{2})^2 .\]

Coercivity follows from Proposition 3.1 in \cite{NWGK2014}; namely, we have that
\begin{align} \label{coerciveEstimate}
\Re A(\u, \u) \geq C_0 ||\u||_{L^2(\Omega)}^2 + \widetilde{C_0}||\nabla\times \u||_{L^2(\Omega)}^2 \geq \min(C_0, \widetilde{C_0}) ||\u||_{\textbf{H}(\text{curl},\Omega)}^2
\end{align}
where 
$$C_0 \coloneqq \inf_{\Omega} \min \left( \text{Eig} \left( -i k_0 (\epsilon (\x) -\epsilon^{\dagger}(\x))\right)\right)$$
and
$$\widetilde{C_0} \coloneqq \inf_{\Omega} \min \left( \text{Eig} \left( i k_0^{-1} (\mu^{-1} (\x) -\left(\mu^{-1}\right)^{\dagger}(\x))\right)\right)$$
where $\text{Eig}$ denotes eigenvalues of the argument. Thus by the Lax-Milgram Theorem, since the form $A(\u, \v)$ is bounded and coercive, there exists a unique solution of (\ref{weakformulation}). 

Now, notice that the linear operator defined by
$$f(\v) \coloneqq \int_{\partial \Omega} \left( \gamma_t^+ (\B_i) -C^e (\gamma_t^+ (\E_i)) \right) \cdot \overline{\pi_-(\v)} \; \mathrm{d} \sigma$$
is bounded from above by
\begin{align*}
&\sup_{||\v||_{\textbf{H}(\text{curl},\Omega)=1}} \left| \int_{\partial \Omega} \left( \gamma_t^+ (\B_i) - C^e (\gamma_t^+ (\E_i))\right) \cdot \overline{\pi_-(\v)} \; \mathrm{d}\sigma \right| \\ &\leq ||C^e||_{\T} \cdot \max(C_{\pi}^1, C_{\pi}^2)  \cdot \left( ||\gamma_t^+ (\B_i)||_{\T} + ||\gamma_t^+ (\E_i)||_{\T} \right) .  
\end{align*}
Thus we see that the solution $\E \in \textbf{H}(\text{curl},\Omega)$ of (\ref{weakformulation}) satisfies
$$||\E||_{\textbf{H}(\text{curl},\Omega)} \leq \dfrac{||C^e||_{\T} \cdot \max(C_{\pi}^1, C_{\pi}^2)}{\min(C_0, \widetilde{C_0})} \left( ||\gamma_t^+ (\B_i)||_{\T} + ||\gamma_t^+ (\E_i)||_{\T} \right)$$
which implies, via Theorem \ref{AlonsoTheorem}, that such $\E$ satisfies the bound
$$||\E||_{\textbf{H}(\text{curl},\Omega)} \leq \dfrac{C_{\Omega} \cdot||C^e||_{\T} \cdot \max(C_{\pi}^1, C_{\pi}^2)}{\min(C_0, \widetilde{C_0})} \left( ||\B_i||_{\textbf{H}(\text{curl},\Omega)} + ||\E_i||_{\textbf{H}(\text{curl},\Omega)}\right)$$
i.e. in terms of the Lipschitz character of the domain:
\begin{align}\label{preciseEestimate}
||\E||_{\textbf{H}(\text{curl},\Omega)} \leq \dfrac{(1+\sqrt{2}) (M\;k_1)^2}{\min(C_0, \widetilde{C_0})}||C^e||_{\T} \left( ||\B_i||_{\textbf{H}(\text{curl},\Omega)} + ||\E_i||_{\textbf{H}(\text{curl},\Omega)}\right),
\end{align}
as required. 

\begin{remark}
In the case of the wedge domain $\Omega_{\alpha}$ from Example \ref{ex:wedge}, for small angles $\alpha$ we see that that $\Hcurl$ norm of $\E$ grows like $\alpha^{-2}$.
\end{remark}

\subsection{Alternative weak formulation}
An alternate weak formulation of this problem can be formulated in terms of the field $\textbf{E}' = \textbf{E}-\textbf{E}_i$, the difference between the internal solution and the incident field; this can be done via the method of \cite{NWGK2014}. 
\begin{theorem}[\cite{NWGK2014}, Theorem 3.2] \label{AlternateTheorem}
Let $\E$ denote the weak solution of (\ref{weakformulation}). Then $\E'$ satisfies the estimate
\begin{align} \label{Theorem32estimate}
||\E'||_{\Hcurl} \leq \dfrac{\max \left( k_0 ||\varepsilon-I_{3\times3}||_{L^\infty(\Omega; \mathbb{C}^{3\times3})}, \; k_0^{-1} ||\mu^{-1}-I_{3\times 3}||_{L^\infty(\Omega; \mathbb{C}^{3\times3})}\right)}{\min(C_0, \widetilde{C_0})} ||\E_i ||_{\Hcurl}, 
\end{align}
where $I_{3\times 3}$ denotes the $3\times 3$ identity matrix.
\end{theorem}

Notice in particular in (\ref{Theorem32estimate}) that there is no contribution from the exterior Calder\'on operator, nor from the Lipschitz constant $M$.

\section{Estimates for the scattered fields} \label{sec:scatteredestimates}
Now that we have quantitative control over the trace mappings, we can obtain precise dependence on the surface of the scattered fields as well. This is particularly useful in cloaking, where one is interested in controlling the scattered fields as much as possible \cite{GKYU2009}. The following estimates have been developed previously; we are now able to make the constants explicit. In particular, we show the constants involve the Lipschitz constant to the power one. 
\begin{theorem}\cite{NWGK2014a}
There exist constants $C_1, C_2>0$ such that the scattered fields on the surface $\partial \Omega$ satisfy
\[ ||\gamma_t^+(\E_s) ||_{\T} \leq C_1 ||\E_i||_{\Hcurl}\]
and
\[||\gamma_t^+(\B_s)||_{\T} \leq C_2 ||\E_i||_{\Hcurl}. \] 
\end{theorem} 
\begin{proof}
From Theorems \ref{AlonsoTheorem}, \ref{AlternateTheorem}, and \ref{Calderonproperties}, we may conclude that these constants are given as 
\begin{align} \label{C1def}
\begin{split}
C_1 &= (1+\sqrt{2})M k_1 \cdot \dfrac{\max \left( k_0 ||\varepsilon-I_{3\times3}||_{L^\infty(\Omega; \mathbb{C}^{3\times 3})},\;  k_0^{-1} ||\mu^{-1}-I_{3\times3}||_{L^\infty(\Omega; \mathbb{C}^{3\times3})}\right)}{\min (C_0, \widetilde{C_0})},  \\ C_2 &= ||C^e||_{\T}\; C_1. 
\end{split}
\end{align}
This is because
\begin{align*}
||&\gamma_t^+(\E_s)||_{\T} =||\gamma_t^-(\E')||_{\T} \leq  (1+\sqrt{2})M  k_1 ||\E'||_{\Hcurl} \leq \\ &\tiny{(1+\sqrt{2})M k_1 \dfrac{\max \left( k_0 ||\varepsilon-I_{3\times3}||_{L^\infty(\Omega; \mathbb{C}^{3\times3})}, \; k_0^{-1} ||\mu^{-1}-I_{3\times 3}||_{L^\infty(\Omega; \mathbb{C}^{3\times3})}\right)}{\min(C_0, \widetilde{C_0})} ||\E_i ||_{\Hcurl}}
 \end{align*}
by Theorem \ref{AlonsoTheorem} and Theorem \ref{AlternateTheorem}. This estimate for $\gamma_t^+(\B_s)$ follows from the fact that
\[ ||\gamma_t^+ (\B_s)||_{\T} = ||C^e(\gamma_t^+(\E_s))||_{\T}\]
and the bound from the Calder\'on operator from Theorem \ref{Calderonproperties}. 
\end{proof} 

Now, there is a well-known integral representation of the scattered field \cite{NWGK2014a}:
\begin{align}\label{integralrep}
\begin{split}
\E_s(x)=\dfrac{\iu}{k_0} \nabla \times \left[ \nabla \times \int_{\partial \Omega} I_{3\times3} g(|x-y|) C^e(\gamma_t^+(\E_s))(y) \; \mathrm{d} \sigma(y) \right] + \\ \nabla \times \int_{\partial \Omega} I_{3\times3} g(|x-y|) \gamma_t^+(\E_s)(y) \; \mathrm{d} \sigma(y), \qquad x\in \Omega_e
\end{split}
\end{align}
where $g(z) = e^{\iu k_0z} / (4\pi z)$ is the fundamental solution of the Helmholtz equation. From this representation coupled with Proposition 4.2 in \cite{NWGK2014}, we obtain the following explicit estimate:

\begin{theorem}
Let $\Omega_s\subset \Omega_e$ be a bounded domain. Then, there exists a constant $C_3>0$ such that 
\begin{align} \label{scatteredestimate1}
||\E_s||_{\left( L^2(\Omega_s)\right)^3} \leq C_3 ||\E_i||_{\Hcurl}
\end{align}
In particular, $C_3= C_1 \cdot \widetilde{C_1}$, where $C_1$ is given in (\ref{C1def})
and $\widetilde{C_1}$ is independent of the Lipschitz constant $M$.
\end{theorem}

\begin{proof}
Let 
\[F_1(x,y) = \dfrac{\iu}{k_0} \nabla \times \left( \nabla \times I_{3\times3} g(|x-y|)\right) \; \text{ and } \; F_2(x,y) = \nabla \times I_{3\times 3}g(|x-y|). \]
From Proposition 4.2 in \cite{NWGK2014}, by tracking constants and using the representation (\ref{integralrep}) we find that
 and 
\[\widetilde{C_1} = \left| \left| ||C^e||_{\T} ||F_1(x, \cdot)||_{\T^{'}} +||F_2(x,\cdot)||_{\T^{'}} \right|\right|_{\left(L^2(\Omega_s)\right)^3}
\]
\end{proof}

Notice also that the dual norm $||\cdot ||_{\T^{'}}$ can be written in terms of the $H^{-1/2}(\text{curl}_{\partial\Omega},\partial\Omega)$ norm using the isomorphism $i_{\pi}^{\ast}$ from Section \ref{sec:dualofT}. The norms of $F_1$ and $F_2$ can be estimated at least outside the smallest circumscribing sphere using spherical harmonics; see \cite{NWGK2014}, Section 4.4. Additionally, in estimating the norms with the isomorphism $i^{\ast}$, there is no dependence on the Lipschitz character of the domain due to how $i_{\pi}^{\ast}$ is constructed.

\section{conclusion}
We have developed a number of quantitative trace estimates in terms of the Lipschitz character of the underlying surface. A weak formulation of the Maxwell system is posed and solved via the Lax-Milgram Theorem, and precise bounds for boundedness and coercivity of the associated sesquilinear form are obtained. Finally, we show that solutions are controlled in $\Hcurl$ by the incident fields and obtain estimates in $\Hcurl$ for the weak solution as well as the scattered fields, which are in terms of constants from continuity of particular trace operators, as well as the material parameters.

\section*{Acknowledgments}
E. S. acknowledges partial support from an American-Scandinavian grant in 2019-2020, and would like to thank FOI for their hospitality where part of this work was completed.

\appendix
\section{Function Spaces} \label{Appendix-functionspaces}
We begin by recalling the basic notions related to fractional Sobolev spaces on bounded, Lipschitz domains. First, if $m \geq 0$ is an integer, recall that $H^m(\Omega)$ denotes the space of distributions on $\Omega$ such that $D^{\alpha}u \in L^2(\Omega)$ for all $|\alpha| \leq m$. When endowed with the norm
$$||u||_{H^m(\Omega)} \coloneqq \left( \sum_{|\alpha| \leq m} \int_{\Omega} |D^\alpha u(x)|^2 dx \right)^{1/2}$$
the space $H^m(\Omega)$ becomes a Banach space. For $s = m + \sigma$ not necessarily an integer, with $0< \sigma< 1$, the spaces $H^s(\Omega)$ can be defined via interpolation (see e.g. \cite{Mclean}) with the norm
$$||u||_{H^s(\Omega)} = \left( ||u||_{H^m(\Omega)}^2+ \sum_{|\alpha|=m} \int_{\Omega} \int_{\Omega} \dfrac{|D^{\alpha}u(x) - D^{\alpha}u(y)|^2}{|x-y|^{n+2\sigma}}dxdy \right)^{1/2}$$ in $\mathbb{R}^n$. Next define the trace space $H^s(\partial \Omega)$ for $0<s<1$ as the space of distributions on $\partial \Omega$ such that
$$\int_{\partial \Omega} |u(x)|^2 d\sigma + \int_{\partial \Omega} \int_{\partial \Omega} \dfrac{|u(x)-u(y)|^2}{|x-y|^{n-1+2s}} d\sigma_x d\sigma_y < +\infty$$
with the norm
$$||u||_{H^s(\partial \Omega)} = \left( \int_{\partial \Omega} |u(x)|^2 d\sigma +\int_{\partial \Omega} \int_{\partial \Omega} \dfrac{|u(x)-u(y)|^2}{|x-y|^{n-1+2s}} d\sigma_x d\sigma_y \right)^{1/2}$$

We define the space
\begin{align*}
\textbf{H}(\text{div}, \Omega) \coloneqq \left\{ U \in \left(L^2 (\Omega)\right)^3: \nabla \cdot U \in L^2 (\Omega) \right\}
\end{align*}
Next we define
$$\textbf{H}^0(\text{div}, \Omega) \coloneqq \text{completion of } \left(C_0^\infty(\Omega)\right)^3 \; \text{ in the } \textbf{H}(\text{div}, \Omega) \; \text{ norm } $$
as well as
\begin{align*}
\textbf{H}(\text{curl}, \Omega) = \left\{ U \in \left(L^2 (\Omega)\right)^3 : \nabla \times U \in \left(L^2 (\Omega)\right)^3 \right\}
\end{align*}
and
$$
\textbf{H}^0(\text{curl}, \Omega) = \text{ completion of } \left(C_0^\infty(\Omega)\right)^3 \text{ in the $\textbf{H}(\text{curl}, \Omega)$ norm}
$$
where the norm in $\textbf{H}(\text{curl}, \Omega)$ and $\textbf{H}(\text{div},\Omega)$ is the graph norm. We recall the following key lemma:
\begin{lemma}[\cite{Lions}] \label{tracethm}
Let $\Omega \subset \RR^3$ be a bounded Lipschitz domain, and let $\nu$ denote the outer unit normal to the boundary $\partial \Omega$.
\begin{enumerate}
\item The following Meyers-Serrin type theorem holds:
$$\textbf{H}(\text{curl}, \Omega) = \overline{ \left(C^\infty(\overline{\Omega})\right)^3 }^{||\cdot ||_{\textbf{H}(\text{curl}, \Omega)}}$$

\item $\textbf{H}^0(\text{curl}, \Omega) = \left\{ U \in \textbf{H}(\text{curl}, \Omega) : (U, \nabla \times \phi)_{L^2} = (\nabla \times U, \phi)_{L^2} \; \forall \; \phi \in \left(C^\infty(\overline{\Omega})\right)^3 \right\}$
\item The space $\textbf{H}^0(\text{curl}, \Omega)$ can also be characterized by those functions in $\textbf{H}(\text{curl}, \Omega)$ having tangential trace zero:
$$\textbf{H}^0(\text{curl}, \Omega) = \left\{ U \in \textbf{H}(\text{curl}, \Omega): \nu \times U = 0 \; \text{ on } \partial \Omega \right\}$$
\end{enumerate}
\end{lemma}

We have a similar result for traces of $\textbf{H}(\text{div}, \Omega)$ functions.

\begin{lemma}[\cite{Lions}] \label{tracetheoremHdiv}
Let $\Omega \subset \mathbb{R}^3$ be a bounded, Lipschitz domain with outer unit normal $\nu$.
\begin{enumerate}
\item The map $\gamma_n$ for $w \in (C^\infty (\overline \Omega))^3$
$$\gamma_n (w) = \left. \nu \cdot w \right|_{\partial \Omega}$$
extends to a continuous linear map from $\textbf{H}(\text{div}, \Omega)$ into $H^{-1/2}(\partial \Omega)$
\item The following characterization holds:
$$\textbf{H}^0(\text{div}, \Omega) = \left\{ V \in \textbf{H}(\text{div}, \Omega): \gamma_n(V) = 0 \right\}$$
\end{enumerate}
\end{lemma}

Recall that the space $\textbf{H}^0(\text{curl}, \Omega)$ can be characterized by those functions having tangential trace equal to zero on the boundary of $\Omega$. But this characterization is not always useful because the trace mapping $\gamma_t : \textbf{H} (\text{curl}, \Omega) \to \left(H^{-1/2}(\partial \Omega)\right)^3$ that sends $w \mapsto \left. \nu \times w \right|_{\partial \Omega}$ is not surjective. Thus, we follow \cite{chen2000} and introduce the trace space $Y(\partial \Omega)$ as
\begin{align*}
Y (\partial \Omega) = \left\{ f \in \left( H^{-1/2}(\partial \Omega) \right)^3: \; \exists \; U \in \textbf{H}(\text{curl}, \Omega) \; \text{ such that } \; \gamma_t (U) = f \right\}
\end{align*}
When it is endowed with the norm
$$||f||_{Y(\partial \Omega)} := \inf_{U \in \textbf{H}(\text{curl}, \Omega); \gamma_t (U) = f } ||U||_{\textbf{H}(\text{curl}, \Omega)}$$
then $Y(\partial \Omega)$ becomes a Banach space. In fact, the following is true.
\begin{theorem}[\cite{buffa2002}]
Let $\Omega \subset \mathbb{R}^3$ be a bounded, smooth domain. Then the space $Y(\partial \Omega)$ as above is a Hilbert space. The trace map $\gamma_t: \textbf{H}(\text{curl}, \Omega) \to Y(\partial \Omega)$ is surjective.
\end{theorem}
Even if $\Omega$ is Lipschitz a similar result is true;  the proof is more involved, but the space $Y(\partial \Omega)$ can be characterized fully. For smooth domains the space $Y(\partial \Omega)$ has a nice characterization. 
To this end we must introduce a few function spaces on $\partial \Omega$. Define

\begin{align} \label{eq_L_t_space}
L^2_t (\partial \Omega) := \left\{ U \in \left(L^2 (\Omega)\right)^3: \nu \cdot U = 0 \; \text{ a.e. on } \partial \Omega \right\}
\end{align}
and
\begin{align*}
H_t^{-1/2}(\partial \Omega) := \left\{ f \in \left( H^{-1/2}(\partial \Omega) \right)^3: f \cdot \nu = 0 \; \text{a.e. on } \partial \Omega \right\}
\end{align*}
Then for smooth $\Omega$ the following holds:

\begin{proposition}[\cite{chen2000}] \label{propy1}
If $\Omega \subset \RR^3$ is a smooth, bounded domain, then
\begin{align} \label{Hminusonehalfdivdef}
Y(\partial \Omega) = \left\{ U \in H_t^{-1/2}(\partial \Omega): \nabla_{\partial \Omega} \cdot U \in H^{-1/2}(\partial \Omega) \right\}
\end{align}
The latter space is denoted $H^{-1/2}(\text{div}, \partial \Omega)$.
\end{proposition}
Recall also that on a smooth domain, the space $TH^{1/2}(\partial \Omega)$ is defined by\footnote{This is one of five equivalent formulations; see \cite{buffa2002}.}
\begin{align} \label{THonehalf}
TH^{1/2}(\partial \Omega) = \left\{ w: \partial \Omega \to \mathbb{R}^3: w \in \left( H^{1/2}(\partial \Omega)\right)^3, \; w \cdot \nu = 0 \right\}
\end{align}

Now we need to define surface differential operators when $\Omega$ is Lipschitz. This can be done by using local coordinates as in \cite{buffa2002}, and the resulting operators $\nabla_{\partial \Omega}: H^1(\partial \Omega) \to L^2_t(\partial \Omega)$ and $\nabla_{\partial \Omega} \times :H^1 (\partial \Omega) \to L^2_t(\partial \Omega)$ are linear and continuous, with adjoints $\nabla_{\partial \Omega} \cdot : L^2_t(\partial \Omega) \to H^{-1}(\partial \Omega)$ and $\text{curl}_{\partial \Omega}: L^2_t(\partial \Omega) \to H^{-1}(\partial \Omega)$, respectively.

Suppose now that $\Omega$ is a Lipschitz domain. In this case there is not such a nice characterization of the space $Y(\partial \Omega)$; however, there is a characterization which is due to L. Tartar \cite{Tartar}. In fact, we define the Tartar trace space on a bounded, Lipschitz domain $\Omega$ by 
\begin{align}\label{eq:def_T-trace_space}
 \textbf{T}(\partial \Omega) &=
 \left\{ \xi \in {\left( H^{-1/2}(\partial \Omega) \right)^3} : \exists \; \eta \in H^{-1/2}(\partial \Omega) \; s.t. \; \forall \; \phi \in H^2 (\Omega), \; \right.\nonumber\\
 &\qquad \left. {} \langle \xi, \gamma(\nabla \phi) \rangle_{H^{1/2}(\partial \Omega)} = \langle \eta, \gamma(\phi) \rangle_{H^{1/2}(\partial \Omega)} \right\}
\end{align}
where $\gamma$ denotes the scalar trace mapping $\gamma: f \mapsto \left. f \right|_{\partial \Omega}$. Then Tartar showed that $\gamma_t: \textbf{H}(\text{curl}, \Omega) \to \textbf{T}(\partial \Omega)$ is surjective.

\end{document}